\documentclass[11pt,a4paper]{article}
\def\cleverefoptions{capitalize,noabbrev}
\usepackage[amsmath,cleveref]{e-jc}
\usepackage{graphicx}
\usepackage{subfig}
\usepackage{tikz}
\usepackage{tkz-berge}
\usepackage{caption,verbatim}
\usepackage{algorithm}
\usepackage{algorithmic}
\usepackage{listings}
\usepackage{hyperref}
\usepackage{url}
\usepackage{setspace}
\usepackage{ragged2e}

\usepackage{soul}
\hypersetup{
	colorlinks=true,
	linkcolor=blue,
	filecolor=blue,
	citecolor = black,
	urlcolor=cyan,
}
\usepackage[utf8]{inputenc}
\usepackage[english]{babel}
\usepackage{indentfirst}
\title{Theoretical and Computational Approaches to Determining Sets of Orders for $(k,g)$-Graphs}

\newtheorem{thm}{Theorem}
\newtheorem{tech}{Technique}
\newtheorem{dfn}[thm]{Definition}

\theoremstyle{definition}

\date{}
\Copyright{The authors. Released under the CC BY-ND license (International 4.0).}
\author{Leonard Chidiebere Eze\authornote{1,3}
\and
Robert Jajcay\authornote{1,3}
\and
Tatiana Jajcayov\'a\authornote{2,3}
\and 
Dominika Z\' avack\' a\authornote{2,3}
}

\authortext{1}{Department of Algebra and Geometry, Faculty of Mathematics, Physics and Informatics, Comenius University, Mlynska Dolina 84248, Bratislava, Slovakia
   (\email{leonard.eze@fmph.uniba.sk}, \email{robert.jajcay@fmph.uniba.sk}).}
\authortext{2}{Department of Applied Informatics, Faculty of Mathematics, Physics and Informatics, Comenius University, Mlynska Dolina 84248, Bratislava, Slovakia (\email{tatiana.jajcayova@fmph.uniba.sk}, \email{dominika.mihalova@fmph.uniba.sk}).}
\authortext{3}{All authors are supported by grants APVV-23-0076, SK-AT-23-0019, and VEGA 1/0437/23.}

\begin{document}

\maketitle
\begin{abstract}
The Cage Problem requires for a given pair $k \geq 3, g \geq 3$ of integers the determination of the order of a smallest $k$-regular graph of girth $g$. We address a more general version of this problem and look for the $(k,g)$-spectrum of orders of $(k,g)$-graphs: the (infinite) list of all orders of $(k,g)$-graphs. By establishing these spectra we aim to gain a better understanding of the structure and properties of $(k,g)$-graphs and hope to use the acquired knowledge in both determining new orders of smallest $k$-regular graphs of girth $g$ as well as developing a set of tools suitable for constructions of extremal graphs with additional requirements. We combine theoretical results with computer based searches, and determine or determine up to a finite list of unresolved cases the $(k,g)$-spectra for parameter pairs for which the orders of the corresponding cages have already been established.
\end{abstract}

\section{Introduction}\label{intro}
\emph{The Cage Problem} is a problem in Extremal Graph Theory that calls for finding a \emph{smallest} 
$k$-regular graph of girth $g$ (with the graphs referred to as \emph{$(k,g)$-graphs}) together with its order $n(k,g)$, for 
all pairs $k \geq 3,g \geq 3$. Even though the problem has been around since the 1960's \cite{karteszi1960piani},
and even though it has received a considerable attention, we only know the exact value of 
$n(k,g)$ and the corresponding smallest $(k,g)$-graphs for a rather slim set of 
parameters \cite{exoo2012dynamic}. The problem is universally acknowledged to be hard, and there is no 
unifying theory behind it. The resolved cases are the result of a number of 
unrelated constructions and proving the extremality of any specific $(k,g)$-graph is in fact the 
core of the complexity of the problem. 
A natural lower bound on the orders of a $(k,g)$-graph is referred to as the \emph{Moore bound $M(k,g)$}, and is expressed as 
 
		\begin{equation}\label{moore}
			M(k,g) = \left\{
			\begin{array}{cc}
				1 + \sum_{i=0}^{\frac{g-3}{2}} k(k-1)^{i}, & g~ \text{odd}, \\ \\
				2\sum_{i=0}^{\frac{g-2}{2}}(k-1)^{i}, & g~ \text{even}. 
			\end{array}
			\right.
		\end{equation}

Maybe specifically due to its complexity, the Cage Problem problem has garnered significant attention 
\cite{exoo2012dynamic}. While much research has focused on finding individual cages, or at least on
finding smaller $(k,g)$-graphs than those found by anybody else (called \emph{record $(k,g)$-graphs}), in our work we take a broader approach and study the entire set of orders of $(k,g)$-graphs for a given $(k,g)$ pair. Given the complexity of the original Cage Problem, our approach might appear futile. Clearly, determining the entire set of orders for a class of graphs for which we cannot even
determine the order of its smallest member does not seem particularly promising. Nevertheless, we 
maintain that taking a broader view and studying the entire class of $(k,g)$-graphs (instead of just
those that stand a chance of being cages) might eventually bring progress even for the question about
the smallest order $n(k,g)$. 

Moreover, the idea of determining all the orders of 
$(k,g)$-graphs goes back to the funding fathers of the topic. Namely, Erd\H{o}s and Sachs \cite{erdos1963exist} proved
already in 1963 that $k$-regular graphs of girth at least $g$ and of order $2m$ exist for all $k \geq 2$, $ g \geq 3 $ and $ m \geq 2 \sum_{t=1}^{g-2}(k-1)^t$. A further improvement by 
Sauer \cite{sauer2extremal}, 1967, pushes the bound further down and specifically asserts the existence of 
$(k,g)$-graphs of all admissible orders greater than or equal to $2(k-1)^{g-2}$, in case of odd $g$, and 
$ 4(k-1)^{g-3}$ for $g$ even (the reader should always keep in mind that the order of an 
odd degree regular graph must be even, and hence, for odd degrees $k$, only even order 
$(k,g)$-graphs might exist). Based on these classical results, one can already conclude that 
for any given pair $ k \geq 3$ and $g\geq 3$ there must exist an order $N$ such that 
a $(k,g)$-graph of order $n$ exists for all $n \geq N$ (all even orders if $k$ is odd).
We shall denote the smallest $N$ with this property by $N(k,g)$, and observe that our choice
of notation implies the non-existence of a $(k,g)$-graph of order $N(k,g)-1$, for even $k$, or order $N(k,g)-2$, for odd $k$.
It is important to note that $N(k,g)$ is not necessarily equal to $n(k,g)$ as is demonstrated
in case of the $(3,8)$-graphs where $n(3,8)$ is known to be equal to $30$, but it has long been known that no $(3,8)$-graph of order $32$ exists. Thus, $ N(3,8) \geq 34 $, and in fact, as 
shown later, $ N(3,8) = 34 $.

Algorithmically, this means that determining the entire spectrum of orders of $(k,g)$-graphs is
essentially a finite task; provided one knows the number $N(k,g)$ and is able to search the 
entire (finite) space of $(k,g)$-graphs of orders smaller than $N(k,g)$. The aim of our paper
is completing this task for specific relatively small pairs $(k,g)$ by developing a set
of tools and constructions possibly applicable for other parameter pairs as well.

Before closing the introductory section, let us mention one more (perhaps far fetched) argument 
for studying the spectra of orders of $(k,g)$-graphs. A $(k,g)$-graph of order $M(k,g)$ 
(of order matching the Moore bound) is called a \emph{Moore graph}, and the parameters of all but one
Moore graph of odd girth have been known since the beginning of the 1960's \cite{hoffman1960moore}.
The parameters of the `missing' Moore graph are $(5,57)$, its order would be $3250$, and to 
this day it is not known whether such a graph exists or not. The clever linear algebra based 
idea that works for the rest of the degree/girth pairs leaves this case unresolved, and
even though a large part of the community might think the graph is impossible (should the graph 
exist, it would have to be in many aspects different from all the other Moore graphs; specifically,
it would have to have a rather small automorphism group \cite{mavcaj2010search}), no one has been able to 
come with an argument against its existence (despite some widely shared announcements). The main
point of the present paragraph is the observation that while one might not be able to show the 
nonexistence of a $(5,57)$-graph of order $3250$, a novel linear algebra or counting (e.g., 
\cite{jajcayova2016counting, jacayova2016improved}) based argument may show the impossibility of the existence of a $(5,57)$-graph of some specific order $n$ larger than $3250$. Should the hypothetical existence of a $(5,57)$-graph of order $3250$
in combination with some of the recursive results developed in this paper (or in the future)
imply the existence of a $(5,57)$-graph of order $n$, such contradiction would 
yield the non-existence of a Moore $(5,57)$-graph.

\section{Basic Concepts and Definitions}\label{sec:basics}

For each pair of parameters $(k,g)$, $k\geq 3$ and $g\geq 3$, the complete set of possible orders of \textit{connected} $k$-regular graphs of girth $g$ will be referred to as \textit{the spectrum of orders of $(k,g)$-graphs}, or \textit{the $(k,g)$-spectrum}. As argued already in the introductory section, 
finding the $(k,g)$-spectrum for a specific pair of parameters $(k,g)$ is extremely difficult, and it requires among other tasks determining the minimum order $n(k,g)$ of the smallest connected $k$-regular graph of girth $g$ (note that a smallest $(k,g)$-graph is necessarily connected).

The concept of the $(k,g)$-spectrum was introduced by Jajcay and Raiman \cite{jajcay2021spectra} and further studied by Eze, Jajcay, and Mih\'alov\'a \cite{eze2023algorithmic}. Results presented in here constitute a continuation of this previous work, building upon the theoretical framework established therein. In particular, several useful  properties of the $(k,g)$-spectra have been proven in \cite{eze2023algorithmic}, including the closeness of the spectra under addition and integral multiples.
These properties simplify the task of determining a specific $(k,g)$-spectrum by limiting the number of orders
$n$ for which one needs to find the corresponding $(k,g)$-graphs.
Typically, only numbers of order up to twice a specific relatively small member of the spectrum 
need to be considered to infer the entirety of the $(k,g)$-spectrum.

Our primary objectives in this paper are to fine-tune the recursive techniques and arguments
from \cite{eze2023algorithmic} that can be employed to determine a small set of orders for which the existence of 
$(k,g)$-graphs has to be resolved, and to employ a computational approach to generate the
missing members of the $(k,g)$-spectra for various families of $(k,g)$-graphs. To accomplish this, we have developed algorithms and methods that enable us to explicitly determine the $(k,g)$-spectrum for a given pair of parameters $(k,g)$. For a start, we concentrate on classes of $(k,g)$-graphs for which the orders of the corresponding $(k,g)$-cages are known. Finding spectra of these families of $(k,g)$-graphs, while feasible, 
still requires extensive searches for $(k,g)$-graphs of specific orders exceeding the order $n(k,g)$. 

Furthermore, since cages are only known for a limited set of parameter pairs, to overcome the challenge of finding suitable starting points for our algorithms, we also explore suitable alternatives. One such approach involves leveraging graphs with desirable properties, such as vertex-transitive graphs, as starting points. In this context, we make use of a comprehensive list of cubic vertex-transitive graphs provided by \cite{potovcnik2013cubic}.

Our paper is organized as follows: Section \ref{sec:Prop} presents properties of $(k,g)$-spectra which play key role in their determination.
Section \ref{sec:Spectra} contains the current status of knowledge about $(k,g)$-spectra for every pair $(k,g)$ for which the cages and their orders $n(k,g)$ are known. These include both complete (completely determined) as well as incomplete $(k,g)$-spectra.
In case of spectra we could not completely determine, we list the orders of graphs for which resolving the question of their existence or non-existence would allow one to completely determine
the corresponding spectrum and determine or present an upper bound on the corresponding number $N(k,g)$. In Section \ref{sec:Comp}, we describe our computational techniques. Finally, we present the conclusions and suggestions for further research in Section \ref{sec:Concl}.

\section{Properties of \texorpdfstring{$(k,g)$}{(k,g)}-Spectra}\label{sec:Prop}
In this section, we discuss techniques for combining $(k,g)$-graphs into larger graphs, which provide the theoretical basis for determining complete $(k,g)$-spectra. This builds on the results of Eze, Jajcay, and Mih\'alov\'a \cite{eze2023algorithmic} where it was established that for any $(k,g)$ pair, two $(k,g)$-graphs can be amalgamated through systematic edge deletions and additions to create a new, expanded $(k,g)$-graph while preserving the key properties of degree and girth.

To mention the possibly easiest example of such amalgamation, consider two (distinct or equal)
connected $(k,g)$-graphs $\Gamma_1, \Gamma_2$ of orders $n_1$ and $n_2$. Let $ u_1v_1 $ be an edge in $\Gamma_1$ and $ u_2v_2 $ be an edge in $\Gamma_2$. The disjoint union of $\Gamma_1$ and $\Gamma_2$ can then be made into a connected graph by removing the edges $u_1v_1$ and $u_2v_2$ and replacing them by either the edges $u_1v_2$ and $u_2v_1 $ or the edges $ u_1u_2$ and $v_1v_2$. In either case, the resulting connected graph is easily seen to be $k$-regular of girth (at least) $g$ and order $n_1+n_2$. In case one is interested in forming graphs of girth exactly equal to  $g$, 
a minor additional assumption will do the trick:
\begin{lemma}[\cite{eze2023algorithmic}]\label{two-graphs}
	Assume that $k \geq 3 $, $g \geq 3 $, and let $ \Gamma_1, \Gamma_2 $ be two $(k,g)$-graphs of
	orders $ n_1 $ and $ n_2 $, respectively, with the additional property that at least one of the two graphs contains
	an edge not contained in a $g$-cycle or that it contains at least two distinct $g$-cycles. 
	Then there exists a $(k,g)$-graph $ \Gamma $ of order $ n_1 + n_2 $.
\end{lemma}

Since it is easy to see that for $k \geq 3 $ and $ g \geq 3 $ any $(k,g)$-graph must 
contain either an edge not contained in a $g$-cycle or at least two distinct $g$-cycles,
the following corollary is obvious. 
\begin{corollary}\label{closed}
	Let $k \geq 3 $ and $ g \geq 3 $. The $(k,g)$-spectrum is closed under addition and
 multiplication by positive integers.
\end{corollary}



As mentioned in the introduction, the only known universal upper bound on the numbers $N(k,g)$ follows from the work
of Sauer \cite{sauer2extremal}
\begin{equation}\label{sauer2extremal}
			N(k,g) \leq \left\{
			\begin{array}{cc}
				2(k-1)^{g-2}, & g~ \text{odd}, \\ \\
				4(k-1)^{g-3}, & g~ \text{even}. 
			\end{array}
			\right.
		\end{equation}
However, based on our computational results, Sauer's bound appears quite a big bigger than the actual values of $N(k,g)$, which 
appear to be close (often equal) to the orders $n(k,g)$ of cages.
The following corollary of Lemma~\ref{two-graphs} and Corollary~\ref{closed} is key to determining $N(k,g)$ for a fixed pair
$(k,g)$, and therefore ultimately to determining the entire $(k,g)$-spectrum. It also makes it clear
why we chose to focus on determining the $(k,g)$-spectra for parameters $(k,g)$ for which $n(k,g)$
has already been determined.
\begin{theorem}[\cite{eze2023algorithmic}]\label{Thm1}
	Let $k \geq 3 $ and $ g \geq 3 $. 
\begin{description}
\item{$(i)$}
	If $k$ is even and there exists an $N$ such that all consecutive integers
	$ N, N+1, N+2, \ldots, N+n(k,g)-1 $ belong to the $(k,g)$-spectrum, then all $n \geq N$ belong to this spectrum, and $ N \geq N(k,g) $.
\item{$(ii)$}
	If $k$ is odd and there exists an even $N$ such that all consecutive even integers $ N, N+2, N+4, \ldots, N+n(k,g)-2 $ belong to the $(k,g)$-spectrum, then all even $n \geq N$ belong to 
	this spectrum, and $ N \geq N(k,g) $.
 \end{description}
\end{theorem} 

Thus, in order to determine the entire spectrum of orders of $(k,g)$-graphs with a known $n(k,g)$, 
the key ingredient not covered by the above theorem is determining the existence or non-existence
of $(k,g)$ graphs of orders between $n(k,g)$ and $2n(k,g)$, and consequently, determining the existence of 
graphs of orders $2n(k,g)+in$, $ i \geq 1 $, for orders $n$ for which $(k,g)$-graphs or order
$n$ do not exist. We address these questions via refining the recursive constructions of $(k,g)$-graphs
from smaller $(k,g)$-graphs to include vertex and/or edge deletion. One example of
such an approach is the following theorem. The Moore trees referred to in its proof are subtrees of the trees 
considered in the usual proof of the Moore bound. 
Namely, given a $(k,g)$-graph $\Gamma$ and its vertex $u$, the subset of 
vertices of $\Gamma$ of distance not exceeding $ r \leq \lfloor \frac{g}{4} \rfloor $ induces a tree of depth
$r$ in which all vertices but the leaves are of degree $k$. We denote this tree by
$\mathcal{T}_{k,r}$, and note that $|V(\mathcal{T}_{k,0})|=1$, $|V(\mathcal{T}_{k,1})|=1+k$,  and  $ |V(\mathcal{T}_{k,r})|=1 + k + k(k-1) + \ldots + k(k-1)^{r-1}$ for $ r \geq 2$. 

 \begin{theorem}\label{n5}
 	Let $k\geq 3$, $g\geq 4$, and $ 0 \leq r \leq \lfloor \frac{g}{4} \rfloor $.
  Then, for every connected $(k,g)$-graph $\Gamma$ of order $n$, there exists a connected $(k,g')$-graph $\Gamma'$ of order $2(n-|V(\mathcal{T}_{k,r})|)$ and girth $g' \geq g$.
 \end{theorem}
 
 \begin{proof}
 Let $\Gamma$ be a connected $(k,g)$-graph of order $n$. Consider two disjoint copies of $\Gamma$, say $\Gamma_1$ and $\Gamma_2$. In each copy $\Gamma_i$, select a vertex $u_i$ together with a Moore tree  $\mathcal{T}_{k,r+1}^i$
 rooted in $u_i$, $i = 1,2$.
 Let $\varphi: V(\mathcal{T}_{k,r+1}^1) \to V(\mathcal{T}_{k,r+1}^2) $ be an isomorphism of rooted trees mapping 
 $u_1$ to $u_2$ and the leaves of $\mathcal{T}_{k,r+1}^1$ to the leaves of $\mathcal{T}_{k,r+1}^2$. Remove the vertices and all their incident edges of the respective Moore subtrees $\mathcal{T}_{k,r}^i$ from both copies $\Gamma_i$, 
 denote the resulting graphs $\Gamma_i^*$, form the disjoint union of the graphs $\Gamma_i^*$, and connect each vertex of degree $k-1$ in the first graph via an edge to its image under $\varphi$ in the second graph (the vertices of degree $k-1$ in each copy are the leaves
 of the corresponding trees $\mathcal{T}_{k,r+1}^i$). Denote the amalgamated graph obtained in this way by $\Gamma^*$. Then, $\Gamma^*$ is clearly $k$-regular of order $2(n-|V(\mathcal{T}_{k,r})|)$. We claim that the girth of
 $\Gamma^*$ is at least as big as the original girth $g$ of $\Gamma$.
 
 To show that, let $v_1^1, v_2^1, \dotsc, v_t^1$, $t=k(k-1)^r$, be the vertices of degree $k-1$ in $\Gamma_1 - \mathcal{T}_{k,r}^1$, and 
 let $v_1^2, v_2^2, \dotsc, v_t^2$ be their respective images under $\varphi $ in $\Gamma_2 - \mathcal{T}_{k,r}^2$.  
 Since the girth of $\Gamma$ is $g$ and the distance between any two of the vertices $v_1^1, v_2^1, \dotsc, v_t^1$
 in $\Gamma_1$ is at most $2r$, their distance in $\Gamma_1 - \mathcal{T}_{k,r}^1$ is at least $g-2r$. The same holds
 true for the vertices $v_1^2, v_2^2, \dotsc, v_t^2$.
 Suppose, for sake of contradiction, that $\Gamma^*$ has girth $g' < g$. As no $g'$-cycles exist within either copy $\Gamma_{i}-\mathcal{T}_{k,r}^i$, any $g'$-cycle in $\Gamma^*$ must include at least two of the added inter-copy edges. Without loss of generality, assume a cycle containing the edges $v_1^1v_1^2$ and $v_i^1v_i^2$. However, the distance between $v_1^1$ and $v_i^1$ as well as the distance between $v_1^2$ 
 and any $v_i^2$ is at least $g-2r$. Thus, $g' \geq 2(g-2r)+2 \geq g $, yielding the desired contradiction.  
 \end{proof}
 
 The use of the above theorem in determining the existence of $(k,g)$-graphs of orders between $n(k,g)$ and $2n(k,g)$ should now be obvious. Selecting the starting graph $\Gamma$ to be a $(k,g)$-cage of order $n(k,g)$ and using
 Theorem~\ref{n5} yields the existence of $k$-regular graphs of orders $ 2(n(k,g)- |V(\mathcal{T}_{k,r})|)$, for all 
 $ 0 \leq r \leq \lfloor \frac{g}{4} \rfloor $; all of them of girths at least $g$. Even though Theorem~\ref{n5} does not 
 guarantee the existence of graphs of girth exactly equal to $g$, inspecting the construction described in its proof
 suggests that the girth of these graphs will be $g$ for the majority of graphs obtained in this way. It is a
 consequence of the empirical observation that most cages
 contain at least one $g$-cycle that does not share vertices with the $\mathcal{T}_{k,r}$ tree even in the case 
 when $r = \lfloor \frac{g}{4} \rfloor$. Notably, while the reader might feel this to be obviously true for all cages, 
 we have been unable to prove such an universal claim.

 We conclude the section with two intriguing `negative' results potentially leading to holes at the very beginning of 
 the corresponding spectra. The \emph{excess} of a $(k,g)$-graph $\Gamma$ is the difference between its order and
 the corresponding Moore bound $M(k,g)$.

  \begin{theorem}[\cite{Big&Ito}]\label{big-ito}
  	 Let $\Gamma$ be a $(k, g)$-cage of girth $g=2m \geq 6$ and excess $e$. If
  	$e\leq k-2$, then $e$ is even and $\Gamma$ is bipartite of diameter $m+1$.\\
  	Let $D(k, 2)$ be the incidence graph of a symmetric $(v, k, 2)$-design.
  	 Let $G$ be a $(k, g)$-cage of girth $g=2m \geq 6$ and excess
  	$2$. Then $g = 6$, $G$ is a double-cover of $D(k, 2)$, and $k$ is not congruent to $5$ or $7 \pmod{8}$.
  \end{theorem}

   \begin{corollary}\label{Tatiana}
   \begin{description}
   \item{$(i)$} Let $k \geq 3 $, and let $g \geq 6$ be even. 
   Then, there exists no $(k,g)$-graph with odd excess $e \leq k-2$.
   \item{$(ii)$} Let $k \geq 3 $, and let $g \geq 8$ be even.
   Then, there exists no $(k,g)$-graph of excess $2$.
   \end{description}
  \end{corollary}

  \section{\texorpdfstring{$(k,g)$}{(k,g)}-Spectra for Parameters with Known \texorpdfstring{$n(k,g)$}{n(k,g)}} \label{sec:Spectra}
  Since determining the order $n(k,g)$ of $(k,g)$-cages for new parameter pairs $(k,g)$ is extremely hard (with
  the last new pairs determined in the 2011 paper \cite{exoo2011computational}), we could not determine the entire spectra of orders of 
  $(k,g)$-graphs for any parameter pairs $(k,g)$ for which $n(k,g)$ is not already known. As explained in the 
  previous section, even when $n(k,g)$ is already known, determining the entire $(k,g)$-spectrum requires finding 
  or showing the non-existence of a (possibly prohibitive) number of graphs of orders between $n(k,g)$ and $2n(k,g)$; 
  and often also of larger orders.
  Theorem~\ref{n5} yields only very few graphs, and the rest need to be found in one way or another. In this 
  section, we report the results of our attempts at determining the entire spectra of orders of graphs where 
  $n(k,g)$ is known. As all orders sought for in this project fall into the range where exhaustive lists of 
  $k$-regular graphs are not feasible, our results are based on highly subjective choices for classes of graphs
  to be considered. Relying on a number of heuristics, we tend to use graphs that exhibit at least some levels 
  of symmetry, graphs whose girths are relatively easy to determine, or graphs listed in searchable databases.
  In cases (quite a few of them) where no databases provided us with examples of specific orders, we relied on 
  searches including a number of random choices followed by exhaustive searches of the remaining possibilities. 
  These were of two types. Either we used graphs of orders close to the ones sought for and
  tried to manipulate them to obtain the desired order or we had to construct the desired examples from scratch.
  While, obviously, the majority of our graphs have been found by the first method, constructing $(k,g)$-graphs 
  of orders close to $2n(k,g)$ is not completely out of question (and we managed to do that for some few cases).
  This is probably due to the empirically observed fact that the further one moves away from the value $n(k,g)$,
  the more graphs of these orders exist, and therefore one stands a better chance of finding such graphs when making 
  some random choices along the way. 

  Despite our best efforts, we have not been able to completely determine the spectra for all parameter pairs
  with known cages. Nevertheless, we list our results for the majority of these pairs while clearly indicating 
  those for which we have obtained the entire spectrum, as well as the orders of graphs we have not been able to find or 
  prove their non-existence for the spectra we were unable to complete.
  
  We begin with cubic graphs where cages are known for $g \in \{3, 4, 5, \dots, 12\}$. We found the entire $(3,g)$-spectra for $g = 3, 4, 5, \dots, 9$, and we were not able to complete our search for $g=10, 11,$ and $ 12$. For these three
  girths, we provide all the orders we have been able to determine and list the orders that proved elusive.
  Recall that in case of an odd degree $k$, the $(k,g)$-spectrum consists exclusively of even numbers.
  
  Throughout our entire text, we adopt the \emph{unified notation $Graph(k,g,n)$ for the $(k,g)$-graphs of order $n$} found using techniques 
  discussed in this text. Should more than one graph with these parameters be considered, 
  we add an additional parameter $i$, and talk about the $Graph(k,g,n;i)$. We only use this notation for new graphs we
  have not found in any previously available data sets, and all the $Graph(k,g,n)$-graphs can be accessed at \cite{EzeJJZ}.

  \subsection{Completely Determined \texorpdfstring{$(3,g)$}{(3,g)}-Spectra}\label{sect3}  
  \begin{description}
  \item{\bf $(3,3)$-Spectrum:}\label{sect30} \sloppy
  It is easy to see that the $(3,3)$-spectrum consists of all positive even integers greater than or equal to $4$
  \[ \{ 4,6,8,10, 12, \ldots \} . \]
  Since the $(3,3)$-cage, $K_4$ is of order $4$, Theorem~\ref{n5} together with Corollary~\ref{closed} and 
  Theorem~\ref{Thm1} yields our claim as long as we can find a $(3,3)$-graph of order $6$. One such graph can
  be obtained by taking two triangles and adding three edges between their vertices to obtain a $3$-regular graph.
  The process is illustrated by the series of graphs included in Fig. \ref{cages3-3}.

 Taking a different approach, starting with the graph $K_4$, we obtained the graphs $Graph(3, 3, 6)$ and $Graph(3, 3, 8)$ by iteratively applying Construction \ref{tec2} described in \Cref{sec:Comp}.

 \begin{figure}[ht]%
		\centering
		\captionsetup[subfloat]{labelformat=empty}
		\subfloat[\centering (a) $(3,3)$-cage]
		{{\includegraphics[width=2.5cm]{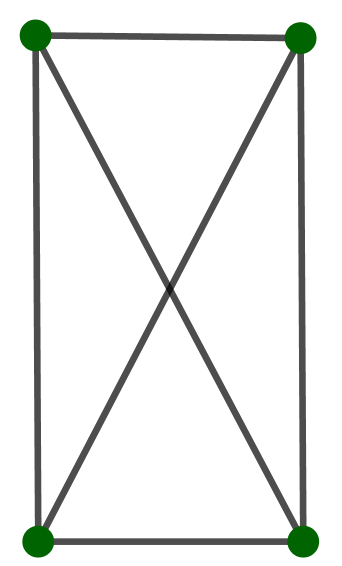} }}
		\captionsetup[subfloat]{labelformat=empty}			
		\subfloat[\centering (b) $(3,3)$-graph of order $6$]
		{{\includegraphics[width=2.5cm]{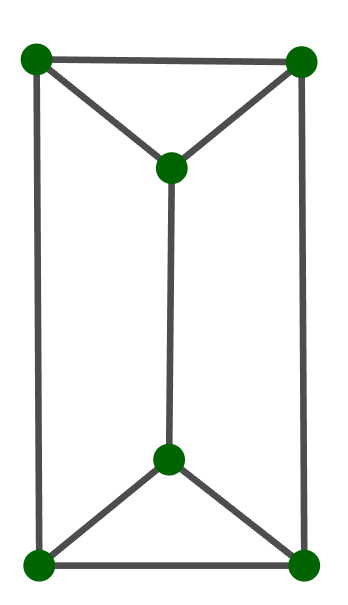} }}
		\captionsetup[subfloat]{labelformat=empty}
		\subfloat[\centering (c) $(3,3)$-graph of order $8$ obtained from $2$ $K_4$]
		{{\includegraphics[width=2.5cm]{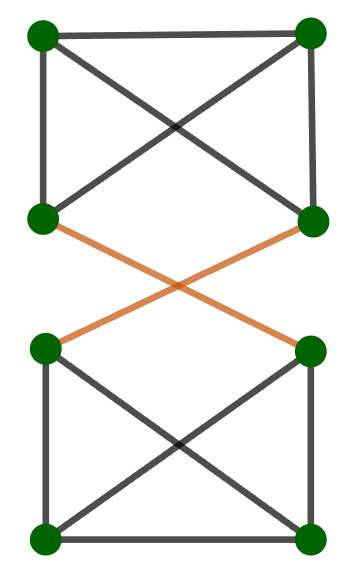} }}
		\captionsetup[subfloat]{labelformat=empty}			
		\subfloat[\centering (d) $(3,3)$-graph of order $10$ obtained from a $K_4$ and $(3,3)$-graph of order $6$]
		{{\includegraphics[width=2.7cm]{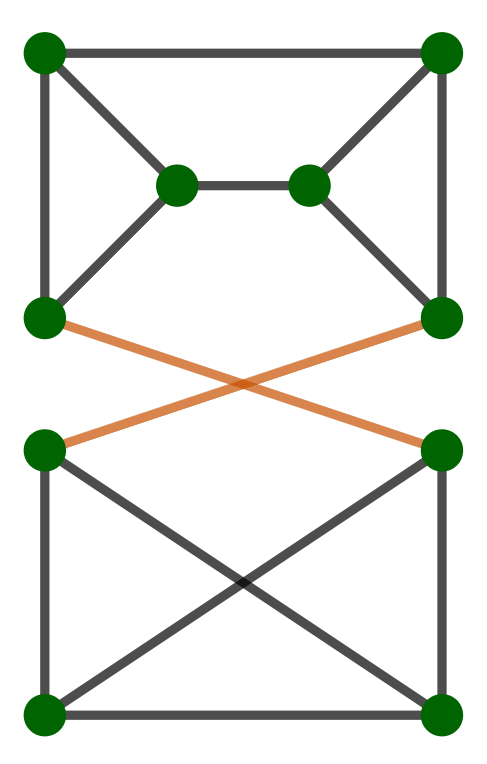} }}
		\caption{Examples of $(3,3)$-Graphs Obtained by Theorem \ref{n5}}
		\label{cages3-3}
	\end{figure}


  \item{\bf $(3,4)$-Spectrum:}\label{sect31}
  The $(3,4)$-spectrum consists of all positive even integers greater than or equal to $6$
  \[ \{ 6,8,10, 12, \ldots \} . \]
  The first element is the order of $K_{3,3}$, the $(3,4)$-cage, which is also equal to $N(3,4)$ and the Moore bound $M(3,4)$. Starting with the graph $K_{3,3}$, we obtained the graphs $Graph(3, 4, 8)$, $Graph(3, 4, 10)$ and $Graph(3, 4, 12)$ by iteratively applying Construction \ref{tec2}.
  The entire spectrum can be generated from $K_{3,3}$ by Construction \ref{tec2} described in Section \ref{sec:Comp}. 
  
  \item{\bf $(3,5)$-Spectrum:}\label{sect32}
  The $(3,5)$-spectrum consists of all positive even integers greater than or equal to $10$ 
  \[ \{10, 12, 14, 16, \dots\}, \] 
  starting from the order of the Petersen graph equal to $N(3,5)$ and $M(3,5)$. Starting with the Petersen graph, we obtained the graphs $Graph(3, 5, 12)$, $Graph(3, 5, 14)$, $Graph(3, 5, 16)$, $Graph(3, 5, 18)$ and $Graph(3, 5, 20)$ by iteratively applying Construction \ref{tec2} (for illustration, consult Figure~\ref{cages}; the graph in
  (c) is isomorphic to the one constructed in (b)). Note that there exist exactly two non-isomorphic $(3,5)$-graphs of order $12$.
  The entire spectrum can be generated by Constructions \ref{tec1} or \ref{tec2}, described in Section \ref{sec:Comp}.  

  \begin{figure}[ht]%
		\centering
		\captionsetup[subfloat]{labelformat=empty}
		\subfloat[\centering (a) $(3,5)$-cage]
		{{\includegraphics[width=3.5cm]{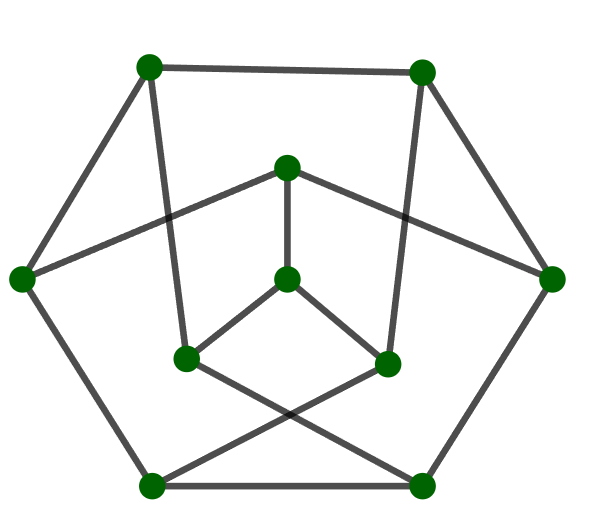} }}
		\captionsetup[subfloat]{labelformat=empty}			
		\subfloat[\centering (b) A subdivision of $(3,5)$-cage]
		{{\includegraphics[width=3.5cm]{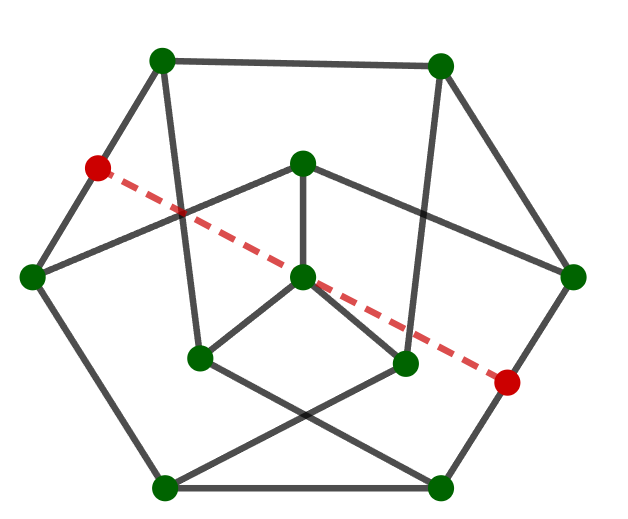} }}
		\captionsetup[subfloat]{labelformat=empty}
		\subfloat[\centering (c) $(3,5)$-graph of order $12$]
		{{\includegraphics[width=3cm]{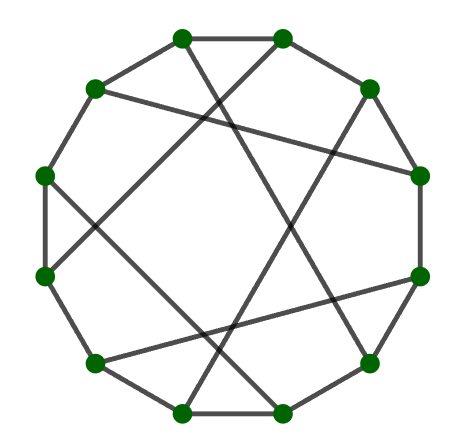} }}
		\captionsetup[subfloat]{labelformat=empty}			
		\subfloat[\centering (d) $(3,5)$-graph of order $14$]
		{{\includegraphics[width=3cm]{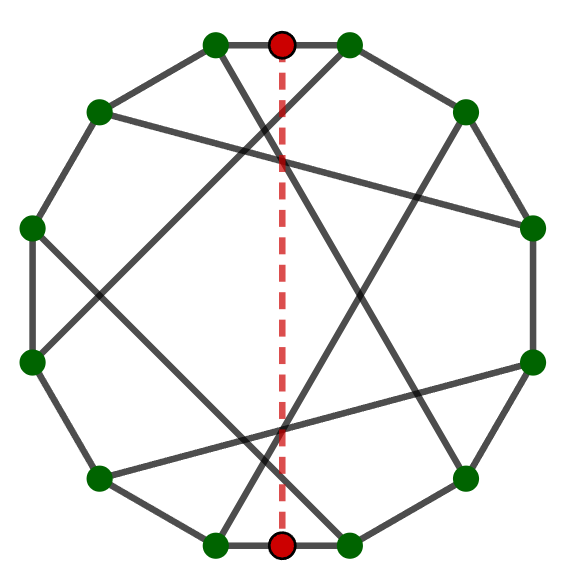} }}
		\caption{Examples of Edge Subdivision Construction }
		\label{cages}
	\end{figure}
  
  \item{\bf $(3,6)$-Spectrum:}\label{sect33}
  The $(3,6)$-spectrum consists of all positive even integers greater than or equal to $14$
  
  \[ \{14, 16, 18, 20, \dots\} , \]
  with the order of the Heawood graph equal to $N(3,6)$ and $M(3,6)$ as its first element. 
  Starting with the Heawood graph, we obtained the graphs $Graph(3, 6, 16)$, $Graph(3, 6, 18)$, $Graph(3, 6, 20)$, $Graph(3, 6, 22)$, $Graph(3, 6, 24)$, $Graph(3, 6, 26)$ and $Graph(3, 6, 28)$ by iteratively applying Construction \ref{tec2}. The entire spectrum can be generated using two different starting graphs: the Petersen graph in combination with Construction \ref{tec1}, or the Heawood graph while relying on Construction \ref{tec2}, as described in Section \ref{sec:Comp}. Returning to Figure~\ref{cages}, we observe that subdivision operation on (c) produces $6$ nonisomorphic graphs; one of which is the Heawood graph. This is the unique way to obtain the Heawood graph by subdividing the (red) edges in Figure~\ref{remark}. 

  \begin{figure}[ht]%
		\centering
		\captionsetup[subfloat]{labelformat=empty}
		\subfloat[\centering (c) $(3,5)$-graph of order $12$]
		{{\includegraphics[width=3.5cm]{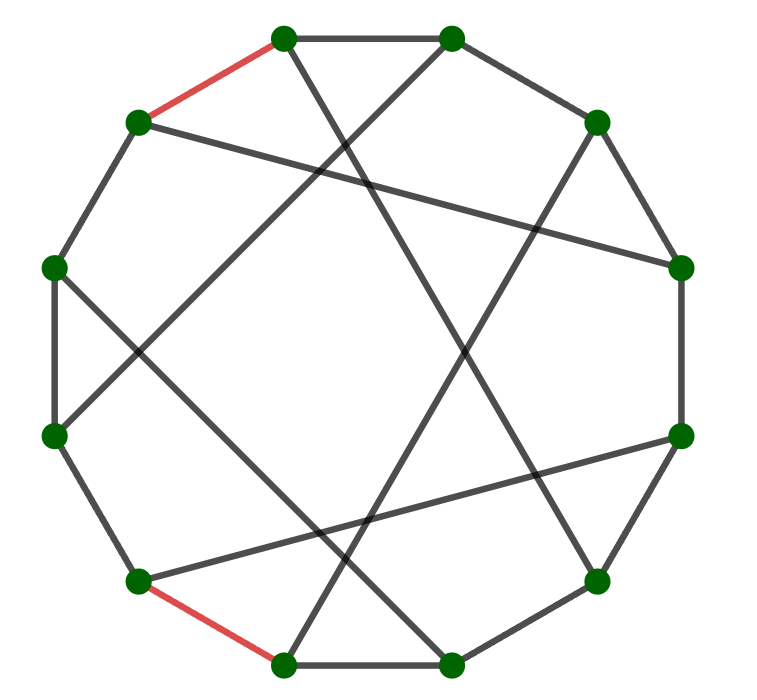} }}
		\captionsetup[subfloat]{labelformat=empty}			
		\subfloat[\centering (e) A subdivision of $(3,5)$-cage]
		{{\includegraphics[width=3.5cm]{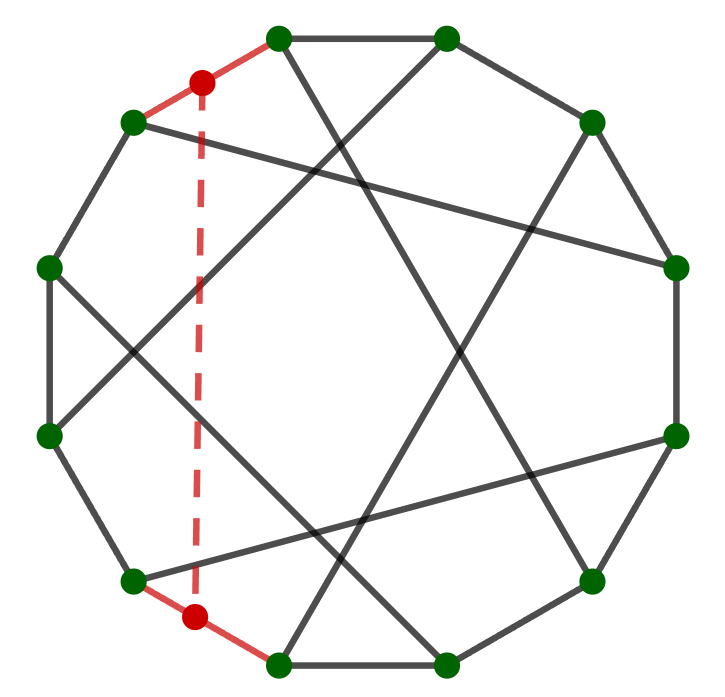} }}
		\caption{A Subdivision of (c) that Yields Heawood Graph }
		\label{remark}
	\end{figure}
    
Note that this is the first spectrum that has not previously appeared explicitly in literature.

  \item{\bf $(3,7)$-Spectrum:}\label{sect34} 
  The $(3,7)$-spectrum consists of all positive even integers greater than or equal to $24$
  \[ \{24, 26, 28, 30, \dots\}, \] 
  beginning with the order of the McGee graph which is also equal to $N(3,7)$, while larger than the Moore bound $M(3,7)$. Using the McGee graph as our starting point, we obtained $Graph(3, 7, 26)$, $Graph(3, 7, 28)$, $Graph(3, 7, 30)$, $Graph(3, 7, 32)$, $Graph(3, 7, 34)$, $Graph(3, 7, 36)$, $Graph(3, 7, 38)$, $Graph(3, 7, 40)$, $Graph(3, 7, 42)$, $Graph(3, 7, 44)$, $Graph(3, 7, 46)$ and $Graph(3, 7, 48)$ by iteratively applying Construction \ref{tec2}. The spectrum is obtained by applying Construction \ref{tec1} to one of the two $(3,7)$-graphs of order 26 found in \cite{housegraph2023} or by using Construction \ref{tec2} described in Section \ref{sec:Comp}.
  
  \item{\bf $(3,8)$-Spectrum:}\label{sect35}
  The $(3,8)$-spectrum consists of $30$ and all positive even integers greater than or equal to $34$ 
  \[ \{30, 34, 36, 38, \dots\}, \]
  starting with the order of the Tutte graph. There is no $(3,8)$-graph of order $32$, creating our first example 
  of a \textit{hole} in the spectrum. The absence of a $(3,8)$-graph of order $32$ was proved by exhaustive computer search \cite{meringer1999fast}. As a result, our techniques cannot generate the spectrum starting from the Tutte $(3,8)$-cage. Due to this, we decided to generate the spectrum beginning with a larger graph of girth $9$ and progressively reduce the orders of the resulting graphs. Specifically, $Graph(3,8,60)$ was obtained from $Graph(3,9,58; 1)$, the first of the $(3,9)$-cages, using Construction \ref{tec2}. Iterative applications of  Construction \ref{tec4} to $Graph(3,8,60)$ lead to the following: $Graph(3,8,58)$, $Graph(3,8,56)$, $Graph(3,8,54)$, $Graph(3,8,52)$, $Graph(3,8,50)$, $Graph(3,8,48;2)$, $Graph(3,8,46;5)$ and $Graph(3,8,44;16)$. The graph $Graph(3,8,48;1)$ was generated from $Graph(3,7,24)$, i.e., the $(3,7)$-cage, via Construction \ref{tec8}. Due to our unsuccessful attempts to find graphs of orders close to $n(3,8)$, the beginning of the spectrum, we decided to conduct an exhaustive iterative search starting from $Graph(3,8,48;1)$ using Construction \ref{tec4}. This process yielded four graphs with parameters $Graph(3,8,46)$, fifteen graphs with parameters $Graph(3,8,44)$, seventeen graphs with parameters $Graph(3,8,42)$, ten graphs with parameters $Graph(3,8,40)$, three graphs with parameters $Graph(3,8,40)$ and one $Graph(3,8,36)$. Additionally, $Graph(3,8,34)$ was obtained from $Graph(3,8,38;1)$ by Construction \ref{tec4},
  The graph $Graph(3,8,42;18)$ is a Cayley graph constructed from the group $\mathbb{S}_3 \times \mathbb{Z}_7$. The graph $Graph(3,8,36;2)$ belongs to the class of Group Divisible Generalized Petersen graphs, denoted as $GDGP_2(18; 5,5)$, described in \Cref{tec6}.

  However, applying Construction \ref{tec1} to a $(3,8)$-graph of order $34$ yields all the required elements of the spectrum. For the first time, $ n(3,8) < N(3,8) = 34 $.

  \item{\bf $(3,9)$-Spectrum:}\label{sect36}
  The $(3,9)$-spectrum consists of all positive even integers greater than or equal to $58$ 
  \[ \{58, 60, 62, 64, \dots\}, \]
  beginning with the order of the $(3,9)$-cages. There are eighteen $(3,9)$-cages, but only one can generate the entire spectrum by applying Construction \ref{tec1}. $N(3,9)$ is the order $n(3,9)$ of the $(3,9)$-cages. Starting with the first $(3,9)$-cage and iteratively applying Construction \ref{tec1}, we obtained the following graphs: $Graph(3, 9, 60)$, $Graph(3, 9, 62)$, $Graph(3, 9, 64)$, $Graph(3, 9, 66)$, $Graph(3, 9, 68)$ and $Graph(3, 9, 70)$. Using Construction \ref{tec1}, the graph $Graph(3,9,74)$ was generated from $Graph(3, 9, 70)$. The graph $Graph(3,9,72)$ was obtained from $Graph(3,9,74)$ through Construction \ref{tec4}. Continuing with $Graph(3,9,74)$, we applied Construction \ref{tec1} to generate the remaining orders of graphs in the spectrum. Additionally, more non-isomorphic graphs with orders from $84$ to $110$ were produced starting from $Graph(3,11,112)$, i.e. the $(3,11)$-cage, using Construction \ref{tec4}: $Graph(3,9,84;2)$, $Graph(3,9,86;2)$, $Graph(3,9,88;2)$, $Graph(3,9,90;2)$, $Graph(3,9,92;2)$, $Graph(3,9,94;2)$, $Graph(3,9,96;2)$, $Graph(3,9,98;2)$, $Graph(3,9,100;2)$, $Graph(3,9,102;2)$, $Graph(3,9,104;2)$, $Graph(3,9,106;2)$, $Graph(3,9,108;2)$ and $Graph(3,9,110;2)$.
  
   \end{description}

  \subsection{Incomplete \texorpdfstring{$(3,g)$}{(3,g)}-Spectra}\label{sect3_1}
 
    \begin{description}
  
       \item{\bf $(3,10)$-Spectrum:}\label{sect37}
  Members of the $(3,10)$-spectrum that we have been able to positively determine include  
  \[ \{70, 72, 80, 82, 84, \dots\} . \] 
  The first $(3,10)$-graph of order $70$ is the cage, whereas the next graph, $Graph(3,10,72)$, belongs to the class of Group Divisible Generalized Petersen graphs, denoted as $GDGP_4(36;7,15,27,19)$ \cite{jasenvcakova2020new, karlubik2018}, described in Section \ref{tec6}. Using the graph of order $80$ found in \cite{potovcnik2013cubic} as a starting graph, we constructed $Graph(3,10,82)$, $Graph(3,10,84)$, $Graph(3,10,86)$, $Graph(3,10,88)$, $Graph(3,10,90)$, $Graph(3,10,92)$, $Graph(3,10,94)$, $Graph(3,10,96)$, $Graph(3,10,98)$ and $Graph(3,10,100)$ through iterative applications of Construction \ref{tec1}. The graph $Graph(3, 10, 102)$ was obtained by Construction \ref{tec5}. Starting from $Graph(3,10, 102)$, we applied Construction \ref{tec1} iteratively to generate graphs with orders from $104$ to $126$.
  The remaining orders of graphs in the spectrum, along with some non-isomorphic graphs, were obtained using $Graph(3,11,144)$ and repeatedly applying Construction \ref{tec4}, we obtained graphs of even orders in the interval $116$
  through $142$.
  Another $(3,10)$-graph of order $80$ can be found in \cite{housegraph2023} and all the following larger graphs can be generated using Construction \ref{tec1} from this graph. It follows that $N(3,10) \leq 80$. 
  
  \emph{The existence or 
  non-existence of $(3,10)$-graphs of orders $ \{74, 76, 78 \} $ is yet to be determined.} The non-existence of at 
  least one of them would give rise to another hole in a cubic spectrum, while the non-existence of at least two of
  these graphs would lead to a novel situation where a hole consists of two consecutive or two non-consecutive values.
  
  \item{\bf $(3,11)$-Spectrum:}\label{sect38} \sloppy
  The orders in the $(3,11)$-spectrum we determined are 
  \[ \{112, 114, 116, 118, 120, 122, 124, 126, 144, \ldots \}. \]
   The $(3,11)$-graph of order $112$ is a cage.
  Starting with the $(3,12)$-cage and applying Construction \ref{tec4} iteratively, we generated: $Graph(3,11,114;1)$, $Graph(3,11,116;1)$, $Graph(3,11,118;1)$, $Graph(3,11,120;1)$, $Graph(3,11,122;1)$ and $Graph(3,11,124;1)$. The $Graph(3,11,126;1)$ was constructed from $Graph(3,11,124;1)$ using Construction \ref{tec1}. Graphs with orders from $114$ to $126$ can also be produced by repeatedly applying Construction \ref{tec1}, starting with the $(3,11)$-cage. Our approach generated two non-isomorphic graphs: $Graph(3,11,114;2)$ and $Graph(3,11,122;2)$. Since further attempts to increase the order of the graphs were unsuccessful, we applied Construction \ref{tec8} to $Graph(3,11,114;1)$ and used Construction \ref{tec4} iteratively to obtain graphs of orders from $198$ to $226$.
 Due to the limited amount of initial graphs, we decided to use cubic vertex-transitive graphs from \cite{potovcnik2013cubic} as starting graphs. Beginning with the graph of order $204$ and girth $12$ \cite{potovcnik2013cubic} and applying Construction \ref{tec4} iteratively, we generated graphs of even orders from
 $182$ to $202$.
 Similarly, starting with the graph of order $182$ and girth $12$ \cite{potovcnik2013cubic} and applying Construction \ref{tec4} repeatedly, we obtained $Graph(3,11,160;1)$, $Graph(3,11,162)$, $Graph(3,11,164)$, $Graph(3,11,166)$, $Graph(3,11,168)$, $Graph(3,11,170)$, $Graph(3,11,172)$, $Graph(3,11,174)$, $Graph(3,11,176)$, $Graph(3,11,178)$ and $Graph(3,11,180)$. We further used both graphs of order $162$ and girth $12$ \cite{potovcnik2013cubic} and iteratively applied Construction \ref{tec4} to generate graphs of orders from $156$ to $160$.
 Using the same starting graphs and applying Construction \ref{tec5}, we obtained $Graph(3,11,148;1)$ and $Graph(3,11,148;2)$. By applying Construction \ref{tec4} to $Graph(3,11,148;1)$, we generated $Graph(3,11,146;1)$. Similarly, by iterative use of Construction \ref{tec4} beginning with $Graph(3,11,148;2)$, we constructed $Graph(3,11,146;2)$ and $Graph(3,11,144;1)$. Starting with $Graph(3,11,148;1)$ and repeatedly applying Construction \ref{tec1}, we obtained: $Graph(3,11,150;1)$, $Graph(3,11,152;1)$, $Graph(3,11,154;3)$, $Graph(3,11,156;3)$, $Graph(3,11,158;3)$ and $Graph(3,11,162;2)$. Likewise, starting with $Graph(3,11,148;2)$ and applying Construction \ref{tec1} iteratively, we generated graphs: $Graph(3,11,150;2)$, $Graph(3,11,152;2)$, $Graph(3,11,154;4)$, $Graph(3,11,156;4)$, $Graph(3,11,158;4)$ and $Graph(3,11,162;3)$. 

 The upper bound on $ N(3,11)$ following from our calculations is the order $144$.
 \emph{The $(3,11)$-graphs of even orders $128$ through $142$ are yet to be found or excluded.} 
  
  \item{\bf $(3,12)$-Spectrum:}\label{sect39} 
  The $(3,12)$-spectrum is not yet fully determined, though some members are known: 
  \[ \{ 126, 162, 168, 180, 182, 192, 204,\ldots \}. \] 
  The smallest element of this spectrum is $126$, the order of the $(3,12)$-cage.
  We constructed a disconnected $(3,12)$-graph of order $252$ (which consists of two disconnected copies of the $(3,12)$-cage) from the $(3,12)$-cage via Construction \ref{tec8}. Starting with this graph, we repeatedly used Construction \ref{tec4} to construct: $Graph(3,12,226)$, $Graph(3,12,228)$, $Graph(3,12,230)$, $Graph(3,12,232)$, $Graph(3,12,234)$, $Graph(3,12,236)$, $Graph(3,12,238)$, $Graph(3,12,240)$, $Graph(3,12,242)$, $Graph(3,12,244)$, $Graph(3,12,246)$, $Graph(3,12,248)$ and $Graph(3,12,250;1)$. Additionally, $Graph(3,12,224;1)$ was generated from the $(3,11)$-cage through Construction \ref{tec8}. By starting with $Graph(3,12,224;1)$ and iteratively applying Construction \ref{tec4}, we obtained the graphs: $Graph(3,12,206)$, $Graph(3,12,208)$, $Graph(3,12,210)$, $Graph(3,12,212)$, $Graph(3,12,214)$, $Graph(3,12,216)$, $Graph(3,12,218)$, $Graph(3,12,220)$ and $Graph(3,12,222)$.
  Four graphs were identified as members of the Group Divisible Generalised Petersen class in \cite{jasenvcakova2020new, karlubik2018}, specifically $Graph(3,12,180)$ as $GDGP_6(90; 11, 41, 29, 77, 47, 71)$, $Graph(3,12,192)$ as $GDGP_3(96;11, 17, 23)$, $Graph(3,12,224; 2)$ as $GDGP_4(112;9,17,65,73)$, and $Graph(3,12,250;2)$ as $GDGP_5(125;11,26,56,106,46)$. Several other graphs with orders $162$, $168$, $182$, $192$, $204$, $216$, $224$, $234$, $240$ and others were found in \cite{potovcnik2013cubic}.  Consequently, $N(3,12) \leq 204$.
  Let us also note that several graphs with the above parameters can be constructed as Group Divisible Generalised Petersen Graphs discussed in Subsection~\ref{tec6}.

 \end{description}
 \begin{table}[ht]
 	\begin{center}
 		\caption{Summary of the results obtained for cubic spectra}\label{taba1}
 		\begin{tabular}{|c|c|c|c|}
 			\hline
 			Girth $g$ & $n(3,g)$ & Orders to be investigated & $N(3,g)$\\
 			\hline
 			3 &  4  &  --  &  4 \\
 			\hline
 			4  &     6   &    -- &  6 \\
 			\hline
 			5  &     10   &    --  &  10 \\
 			\hline
 			6  &   14     &     -- &  14 \\  
 			\hline
 			7  &   24    &   --  & 24 \\
 			\hline
 			8  &    30   &    -- &  34 \\
 			\hline
 			9  &    58  &   --  & 58 \\
 			\hline
 			10  &   70  &   74, 76, 78  & $\leq$ 80 \\
 			\hline
 			11  &   112  &   128 -- 142 & $\leq$ 144 \\
 			\hline
 			12 &    126  &  128 -- 160, 164, 166, 170 -- 178, 184 -- 190, 194 -- 202  & $\leq$ 204 \\
 			\hline
 		\end{tabular}
 	\end{center}
 \end{table}

 \subsection{Completely Determined \texorpdfstring{$(4,g)$}{(4,g)}-Spectra}\label{sect4} 
 \begin{description} 
\item {\bf $(4,3)$-Spectrum:}\label{sect4_3}
  The $(4,3)$-spectrum consists of all positive integers greater than or equal to $5$: 
  \[ \{5, 6, 7, 8, \dots\}, \]   
  starting with the order of the $(4,3)$-cage, the complete graph $K_5$. Starting with the $Graph(4,4,11)$, we repeatedly applied Construction \ref{tec4} to generate the required graphs: $Graph(4,3,6)$, $Graph(4,3,7)$, $Graph(4,3,8)$, $Graph(4,3,9)$ and $Graph(4,3,10)$.
 
  \item {\bf $(4,4)$-Spectrum:}\label{sect40}
  The $(4,4)$-spectrum consists of 8 and all positive integers greater than or equal to 10 
  \[ \{8, 10, 11, 12, \dots\}, \]   
  starting with the order of the $(4,4)$-cage. Notably, there is no $(4,4)$-graph of order 9 \cite{meringer1999fast}, marking the first instance of a \textit{hole} in a $(4,g)$-spectrum. We obtained $(4,4)$-graphs of order 10 and above using the Circulant Construction \ref{circ1} described in Section~\ref{sec:Comp}. These include $Graph(4,4,10)$, $Graph(4,4,11)$, $Graph(4,4,12)$, $Graph(4,4,13)$, $Graph(4,4,14)$, $Graph(4,4,15)$ and $Graph(4,4,16)$.
  All members of the $(4,4)$-spectrum were obtained as circulant graphs described in Subsection \ref{circ1}. As
  mentioned before, $n(4,4) < N(4,4) = 10$, as for $n=9$ the construction yields a graph of girth $3$.
  
  \item {\bf $(4,5)$-Spectrum:}\label{sec:Comp2}
  The $(4,5)$-spectrum consists of all positive integers greater than or equal to 19  
  \[ \{19, 20, 21, \dots\}, \]  
  starting with the order of the $(4,5)$-cage, known as the Robertson graph. Although a $(4,5)$-graph of order $20$ can be found in \cite{housegraph2023}, we also obtained $Graph(4,5,20)$ using Construction \ref{tec2}. The $(4,5)$-graph of order $21$ is the Brinkmann graph, as documented in \cite{housegraph2023}. 
  Starting with $Graph(4,6,39)$, we iteratively applied Construction \ref{tec4} to generate graphs: $Graph(4,5,27;1)$, $Graph(4,5,28)$, $Graph(4,5,29)$, $Graph(4,5,30)$, $Graph(4,5,31)$, $Graph(4,5,32)$, $Graph(4,5,33)$, $Graph(4,5,34)$, $Graph(4,5,35)$, $Graph(4,5,36)$, $Graph(4,5,37)$ and $Graph(4,5,38)$. $Graph(4,5,23)$, $Graph(4,5,24)$, $Graph(4,5,25)$, $Graph(4,5,26)$ and $Graph(4,5,27;2)$ were constructed by Construction \ref{tec4} from the $Graph(4,6,28)$. $Graph(4,5,22)$ was obtained by Construction \ref{tec4} from graph $Graph(4,5,24)$.
   Thus, $N(4,5) = n(4,5)$.
  
\end{description}

 \subsection{Incomplete \texorpdfstring{$(4,g)$}{(4,g)}-Spectra}
\begin{description}
    
  \item {\bf $(4,6)$-Spectrum:}\label{sec:Comp3}
  The known members of the $(4,6)$-spectrum include all positive even integers greater than or equal to $26$ and all positive odd integers greater than or equal to $39$:  
  \[ \{26, 28, 30, 32, 34, 35, 36, 38, 39, 40, 41, \dots\}. \]  
  The $(4,6)$-spectrum begins with $26$, the order of a $(4,6)$-cage. By Corollary \ref{Tatiana} $(i)$, no $(4,6)$-graph exists for order $27$, and exhaustive searches in \cite{meringer1999fast} exclude the existence of $(4,6)$-graphs of orders $27, 29, 31$, and $33$. The $(4,6)$-graph of order $35$ was found in \cite{potovcnik2013cubic, potocnik2015tetravalent}. \emph{The existence of a $(4,6)$-graph of order $37$ is currently undecided.} Nevertheless, the 
  $(4,6)$-spectrum is already remarkably complicated, and contains a number of non-consecutive holes; all of them odd.
  All graphs of even order from $26$ onward were obtained using Circulant Construction \ref{circ1}. 
  This construction is performed independently, without the requirement for a pre-existing graph. 
  
  Odd-order graphs from $41$ onward were constructed from the circulant graphs of even orders greater than or equal to $40$ used as starting graphs for the use of Construction \ref{tec2}.
  $Graph(4,6,39)$ was obtained by Construction \ref{tec4} from $Graph(4,6,41)$. Our results establish the upper
  bound $N(4,6) \leq 38$.

  \item {\bf $(4,7)$-Spectrum:}\label{sec:Comp4}
  The $(4,7)$-spectrum is not fully determined, though some members are known:  
  \[ \{67, 68, 70, 72, 73, 74, 75, \ldots, 80, 84, 85, 86, 87, \dots\}. \]  
  \emph{The existence of $(4,7)$-graphs of orders $69$, $71$, and $81$ to $83$ is still undecided.} The $(4,7)$-spectrum begins with number $67$, the order of the $(4,7)$-cage. We obtained $Graph(4,7,68)$ by applying Construction \ref{tec2} to the $(4,7)$-cage. Using iteratively Construction \ref{tec4} starting from the $(4,8)$-cage of order $80$, we obtained $Graph(4,7,73)$, $Graph(4,7,74)$, $Graph(4,7,75)$, $Graph(4,7,76)$, $Graph(4,7,77)$, $Graph(4,7,78)$ and $Graph(4,7,79)$. From $Graph(4,7,79)$, the one-vertex-larger $Graph(4,7,80)$ was obtained by Construction \ref{tec1}. Similarly, we obtained $Graph(4,7,70)$ by Construction \ref{tec4} from $Graph(4,7,73)$.
  
  We obtained graphs $Graph(4,7,90)$, $Graph(4,7,91)$, $Graph(4,7,92)$, $Graph(4,7,93)$, $Graph(4,7,94)$ and $Graph(4,7,95)$ from the $(4,8)$-graph of order $96$ in \cite{potocnik2015tetravalent}
   by recursively applying Construction \ref{tec4}. Using $Graph(4,7,90)$, we obtained graphs $Graph(4,7,84;1)$, $Graph(4,7,86)$, $Graph(4,7,88)$ by successively applying Construction \ref{tec4}. Additionally, $Graph(4,7,85)$, $Graph(4,7,87)$, $Graph(4,7,89)$ were constructed by Construction \ref{tec1} from graphs $Graph(4,7,84;1)$, $Graph(4,7,86)$, $Graph(4,7,88)$, respectively. Starting with $Graph(4,7,95)$ and iteratively applying Construction \ref{tec1}, we generated $(4,7)$-graphs of orders starting with $96$ and finishing with $134$.
  $Graph(4,7,72)$ and $Graph(4,7,84;2)$ are Cayley graphs. $ N(4,7) \leq 84 $.
  
  
  \item {\bf $(4,8)$-Spectrum:}\label{sec:Comp5}
  The known members of the $(4,8)$-spectrum include  
  \[ \{80, 96, 100, 108, 110, 120, 124, 126, 128, 130, 131, \ldots \}. \]
  The order of the $(4,8)$-cage is $80$. The $(4,8)$-graph of order $81$ does not exist due to Corollary \ref{Tatiana}$(i)$ and the graph of order $82$ does not exist due to Corollary \ref{Tatiana}$(ii)$. The $(4,8)$-graphs of orders $96$, $100$, $110$ and others were identified in the list of edge-transitive and arc-transitive tetravalent graphs in \cite{potocnik2015tetravalent}. The $(4,8)$-graphs of orders $108$, $120$ are Cayley graphs. Starting from the $(4,8)$-graph of order $146$ constructed from $Graph(4,7,73)$ via Construction \ref{tec8} and using recursive application of Construction \ref{tec1}, graphs of orders $147$ through $160$ were obtained.
  The $Graph(4,8,134)$ was obtained by Construction \ref{tec8}. The set of graphs: $Graph(4,8,135)$, $Graph(4,8,136)$, $Graph(4,8,137)$, $Graph(4,8,138)$, $Graph(4,8,139)$, $Graph(4,8,140)$, $Graph(4,8,141)$, $Graph(4,8,142)$, $Graph(4,8,143)$, $Graph(4,8,144)$, $Graph(4,8,145)$ and $Graph(4,8,146)$, was generated by iteratively applying Construction \ref{tec1}. Starting with  $Graph(4,8,134)$ and Construction \ref{tec4}, we constructed $Graph(4,8,124)$, $Graph(4,8,126)$, $Graph(4,8,128)$, $Graph(4,8,130)$ and $Graph(4,8,132)$. By Construction \ref{tec4}, we obtained $Graph(4,8,131)$ and $Graph(4,8,133)$ from $Graph(4,8,130)$ and $Graph(4,8,132)$, respectively.  
\end{description}

 \begin{table}[ht]
 	\begin{center}
 		\caption{Summary of the results obtained for $(4,g)$-spectra}\label{taba2}
 		\begin{tabular}{|c|c|c|c|}
 			\hline
 			Girth $g$ & $n(4,g)$ & Orders to be investigated & $N(4,g)$ \\
 			\hline
 			4 &  8  &  --  &  10  \\
 			\hline
 			5  &     19   &    -- &  19 \\
 			\hline
 			6  &     26   &    37  &  $\leq$ 38\\
 			\hline
 			7  &   67     &    69, 71, 81--83 &  $\leq$ 84\\  
 			\hline
 			8  &   80    &   83--95, 97--99, 101--107, 109, 111--119, 121--123, 125, 127, 129  & $\leq$ 130 \\
 			\hline
 		\end{tabular}
 	\end{center}
 \end{table}

 \subsection{Completely Determined \texorpdfstring{$(5,g)$}{(5,g)}-Spectra}\label{sect5}
 
 \begin{description} 
 
 \item {\bf $(5,3)$-Spectrum}\label{sect51}
 The $(5,3)$-spectrum consists of all positive even integers greater than or equal to $6$ 
 \[ \{6, 8, 10, \dots,\}, \] beginning with $6$, the order of the $(5,3)$-cage $K_6$. $N(5,3)$ is equal to the Moore bound $M(5,3)$. Starting with $Graph(5,4,14)$, we repeatedly applied Construction \ref{tec4} to generate the required graphs $Graph(5,3,8)$, $Graph(5,3,10)$ and $Graph(5,3,12)$.

 \item {\bf $(5,4)$-Spectrum:}\label{sect52}
 The $(5,4)$-spectrum consists of all positive even integers greater than or equal to 10
 \[ \{10, 12, 14, \dots\}, \] starting with order of the $(5,4)$-cage $K_{5,5}$. A $(5,4)$-graph of order $20$ was obtained via Construction \ref{tec8} from the $(5,4)$-cage. Starting from the graph of order $20$, we repeatedly used Construction \ref{tec4} to construct the following graphs: $Graph(5,3,12)$, $Graph(5,3,14)$, $Graph(5,3,16)$ and $Graph(5,3,18)$. $N(5,4)$ is equal to the Moore bound $M(5,4)$.
 

  \item {\bf $(5,6)$-Spectrum:}\label{sect53}
 The $(5,6)$-spectrum consists of the order of the $(5,6)$-cage, $42$, and after skipping $44$, continues with all positive even integers greater than or equal to 46
 \[ \{42, 46, 48, 50, 52, \ldots \}. \] There is no $(5,6)$-graph of order $44$. All orders of the spectrum from $46$ to $84$ can be realized via Cayley graphs. $N(5,6) = 46$.
 
 \end{description}
\subsection{Incomplete \texorpdfstring{$(5,g)$}{(5,g)}-Spectra}\label{sect5_1}
 
 \begin{description} 
 \item {\bf $(5,5)$-Spectrum:}\label{sect54} \sloppy
  The known members of the $(5,5)$-spectrum include  
  \[ \{30, 32, 36, 38, 40, 42, \ldots \}. \]  
 \emph{The only potential member of the $(5,5)$-spectrum not yet decided is $34$}. The $(5,5)$-spectrum starts with the order of the $(5,5)$-cage, which is $30$. $N(5,5) \leq 36$. Observe that the Moore bound $M(5,5)$ is $26$, which is less than $n(5,5)$. The $(5,5)$-graphs $Graph(5,5,36)$, $Graph(5,5,38)$, and $Graph(5,5,40)$ were obtained by recursively applying Construction \ref{tec4} to the $(5,6)$-cage of order $42$. Starting with the $(5,6)$-graph of order $60$ constructed from the $(5,5)$-cage using Construction \ref{tec8}, we generated $Graph(5,5,58)$, $Graph(5,5,56;1)$, $Graph(5,5,54;1)$, $Graph(5,5,52)$, $Graph(5,5,50;1)$, $Graph(5,5,48;1)$, $Graph(5,5,46)$, $Graph(5,5,44)$ and $Graph(5,5,42;1)$.
 Of our list, the following are Cayley graphs: $Graph(5,5,32)$, $Graph(5,5,42;2)$, $Graph(5,5,48;2)$, $Graph(5,5,50;2)$, $Graph(5,5,54;2)$, $Graph(5,5,56;2)$, $Graph(5,5,56;3)$, $Graph(5,5,60;1)$, $Graph(5,5,60;2)$ and $Graph(5,5,60;3)$.
  \end{description}

  \begin{description} 
 \item {\bf $(5,7)$-Spectrum:}\label{sect57}
 The known members of the $(5,7)$-spectrum include  
  \[ \{152, 154, \ldots, 168, 260, 262, \ldots \}. \]  
 The order of the $(5,7)$-cage is $152$. By iteratively applying Construction \ref{tec4} to the $(5,8)$-cage, we generated graphs of even orders from $154$ to $168$. Additionally, applying Construction \ref{tec8} to the $(5,7)$-cage resulted in a $(5,8)$-graph of order $304$. Starting from this graph and repeatedly using Construction \ref{tec4}, we obtained graphs of even orders from $278$ to $302$. The final set of graphs was produced by applying Construction \ref{tec4} to the $(5,8)$-graph of order $288$ \cite{potocnikPentavelent}. The resulting graphs include: $Graph(5,7,260)$, $Graph(5,7,262)$, $Graph(5,7,264)$, $Graph(5,7,266)$, $Graph(5,7,268)$, $Graph(5,7,270)$, $Graph(5,7,272)$, $Graph(5,7,274)$, $Graph(5,7,276)$, $Graph(5,7,278;2)$, $Graph(5,7,280;2)$, $Graph(5,7,282;2)$, $Graph(5,7,284;2)$, and $Graph(5,7,286;2)$.
  \end{description}
 
   \begin{table}[ht]
 	\begin{center}
  \caption{Summary of the results obtained for $(5,g)$-spectra}\label{taba3}
 		\begin{tabular}{|c|c|c|c|}
 			\hline
 			Girth $g$ & $n(5,g)$ & Orders to be investigated & $N(5,g)$ \\
 			\hline
 			3 &  6  &  --  &  6  \\
 			\hline
 			4 &  10  &  --  &  10  \\
 			\hline
 			5  &     30   &    34 & $ \leq$ 36 \\
 			\hline
 			6  &     42   &   --   &  46\\
 			\hline
            7  &     152   &   170--258   &  260\\
 			\hline
 		\end{tabular}
 	\end{center}
 \end{table}
 
 \subsection{Other Completely Determined \texorpdfstring{$(k,g)$}{(k,g)}-Spectra}\label{sec:Comp10}
 
While our primary focus has been on $(3,g)$-, $(4,g)$-, and $(5,g)$-spectra, we have also determined some other $(k,g)$-spectra including the $(6,3)$-, $(6,4)$-, $(7,3)$, and $(7,4)$-spectrum.

\begin{description}
    \item {\bf $(6,3)$-Spectrum:} \label{sect71}
    The $(6,3)$-spectrum is the set of all positive integers greater than or equal to $7$    
    \[ \{7, 8, 9, \ldots \},\] beginning with the order $7$ of $K_7$. Here, $N(6,3) = 7$, which is equal to the Moore bound $M(6,3)$. We obtained $Graph(6,3,8)$, $Graph(6,3,9)$, $Graph(6,3,10)$, $Graph(6,3,11)$, $Graph(6,3,12)$, $Graph(6,3,13)$ and $Graph(6,3,14)$ by iteratively applying Construction \ref{tec4} starting from $Graph(6,4,16)$.

\item {\bf $(6,4)$-Spectrum:}\label{sect72}
The $(6,4)$-spectrum consists of 12 and all positive integers greater than or equal to 14
   \[ \{12, 14, 15, 16, \ldots \}.\] The first element of the spectrum is the order of the $(6,4)$-cage $K_{6,6}$.
   Notably, there is no $(6,4)$-graph of order $13$ \cite{meringer1999fast}, creating a hole in the spectrum. Thus, the $N(6,4) = 14$. All members of the spectrum greater than or equal to $14$ were obtained by exhaustive computer search \cite{meringer1999fast}.
   Graphs of even order can also be constructed from $(6,3)$-graphs by Construction \ref{tec8}. $Graph(6,4,25)$, $Graph(6,4,23)$, $Graph(6,4,21)$ were obtained by iteratively applying Construction \ref{tec4} to the $(6,4)$-graph of order $27$ found in \cite{housegraph2023}.
   Starting with a $(6,4)$-graph of order $19$ found in \cite{housegraph2023} and using Construction \ref{tec4}, we obtained $Graph(6,4,17)$. Other members of the spectrum can be also found in \cite{housegraph2023}.

\item {\bf $(7,3)$-Spectrum:}\label{sect73} 
The $(7,3)$-spectrum is the set of all positive even integers greater than or equal to $8$
   \[ \{8, 10, 12, \ldots \}.\] 
   The first element, $8$, is the order of the complete graph $K_8$ on $8$ vertices. The value $N(7,3) = 8$. We 
   constructed $Graph(7,4,16)$ from the $(7,3)$-cage by Construction \ref{tec8}. Graphs of orders $12$ and $14$ were generated from $Graph(7,4,16)$ by iteratively applying Construction \ref{tec4}. A $(7,3)$-graph of order $10$ can be found in \cite{housegraph2023}.
   
\item {\bf $(7,4)$-Spectrum:}\label{sect74} 
The $(7,4)$-spectrum consists of all positive even integers greater than or equal to $14$
     \[\{14, 16, 18, \ldots \}. \]
  It begins with $14$, which is the order of the $(7,4)$-cage $K_{7,7}$. The value of $N(7,4)$ is $14$. 
  To complete the spectrum, we obtained a $(7,4)$-graph of order $28$ from the $(7,4)$-cage by Construction \ref{tec8}. Graphs of orders $16$ to $26$ were generated from this disconnected graph by iteratively applying Construction \ref{tec4}.

\end{description}

  \begin{table}[ht]
 	\begin{center}
 		\caption{Summary of the results obtained for $k=6$}\label{taba4}
 		\begin{tabular}{|c|c|c|c|}
 			\hline
 			Girth $g$ & $n(6,g)$ & Orders to be investigated & $N(6,g)$ \\
 			\hline
 			3 &  7  &  --  &  7  \\
 			\hline
 			4 &  12  &  --  &  14  \\
 			\hline  		
 		\end{tabular}
 	\end{center}
 \end{table}
 
 \begin{table}[ht]
 	\begin{center}
 		\caption{Summary of the results obtained for $k=7$}\label{taba5}
 		\begin{tabular}{|c|c|c|c|}
 			\hline
 			Girth $g$ & $n(7,g)$ & Orders to be investigated & $N(7,g)$ \\
 			\hline
 			3 &  8  &  --  &  8  \\
 			\hline
 			4 &  14  &  --  &  14  \\
 			\hline
 		\end{tabular}
 	\end{center}
 \end{table}

 \section{Computational Techniques}\label{sec:Comp}
 Having surveyed our results which constitute the current state of the art regarding $(k,g)$-spectra, we now focus on computational techniques enabling us to acquire the information provided in the previous section. We outline the various algorithms and methods we have developed and employed. These techniques range from edge and vertex manipulations to arithmetic and algebra based constructions involving or mimicking circulants and Cayley graphs. Our methods are algorithmically relatively straightforward and did not require careful optimization. Each method has its strengths and limitations, and despite their simplicity, it was their combined application that allowed us to make progress in mapping out previously unknown $(k,g)$-spectra. One of the reasons behind the effectiveness of our methods probably also lies in the intuitive 
 observation that the number of $(k,g)$-graphs grows quickly with the increase of the considered order, and therefore
 finding $(k,g)$-graphs whose order is close to $2n(k,g)$ is not particularly demanding. On the other hand, graphs
 of order close to $n(k,g)$, which are most likely more scarce, can often be constructed from the corresponding
 cage.
 Some of our techniques yield the entire $(k,g)$-spectrum for specific values of $k$ and $g$ alone, and when a single technique was insufficient, we combined two or more techniques to achieve the desired results. 
 
  \subsection{Edge Subdivision Constructions}\label{tec1}
  There are various kinds of edge subdivision in the literature \cite{brinkmann2017generation, jajcay2021spectra}. Here, we discuss only two-edge subdivision and triple-edge subdivision.
 
 \begin{dfn}
 	 Let $\Gamma$ be a $(k,g)$-graph of order $n$ and let $e_1$ and $e_2$ be two edges of $\Gamma$. Suppose $e_1=v_1v_2$ and $e_2=w_1w_2$. We define the edge distance $E(e_1,e_2)$ between $e_1$ and $e_2$ as follows:
  
 	\begin{equation}
 		E(e_1,e_2) = min\{d(v_1,w_1), d(v_1,w_2), d(v_2,w_1), d(v_2,w_2)\} + 1,
 	\end{equation}
 	where $d(v,w)$ denotes the usual distance metric between vertices
    $$d(v,w) = min\{length(p):p~ is~ a ~path ~from~ u ~to~ w\}.$$
 	
 	 \end{dfn}
 
 \begin{tech}\label{tech1}
 	Let $\Gamma$ be a cubic $(3,g)$-graph of order $n$. Let $e_1$ and $e_2$ be two edges of $\Gamma$ such that the edge distance $E(e_1,e_2)$ is at least $g-2$. Introduce a vertex (subdivide) on each of the edges $e_1$ and $e_2$ and join the two new vertices with an edge to obtain a new cubic graph $\Gamma'$ with two more vertices and one more
    edge.
 \end{tech}
 
   \begin{tech}\label{tech2}
  	Let $\Gamma$ be a cubic $(k,g)$-graph of order $n$. Let $e_1$, $e_2$, and $e_3$ be three edges such that the pairwise edge distance between  $e_1$, $e_2$, and $e_3$ is at least $g-3$. Introduce a vertex on (subdivide) each of the edges $e_1$, $e_2$, and $e_3$ and join the three introduced vertices via another new vertex to obtain a new graph $\Gamma''$.
  \end{tech}

 	It is easy to see that if $\Gamma$ is a cubic graph of order $n$ and girth $g$, $\Gamma'$ obtained using Technique~\ref{tech1} is a cubic graph of order $n+2$ and girth greater than or equal to $g$. Graph $ \Gamma''$ obtained from $\Gamma$ using Technique~\ref{tech2} is also cubic, of order $n+4$, and girth greater than or equal to $g$. These techniques alone were sufficient for generating the $(3,5)- (3,9)$-spectra (which only consist of even
    orders).

    The next technique allows for adding a single vertex to tetravalent graphs; which can also have odd orders.

\begin{tech}\label{tech-tetra}
  	Let $\Gamma$ be a tetravalent $(k,g)$-graph of order $n$. Let $e_1$ and $e_2$ be two edges of $\Gamma$
    whose edge distance $E(e_1,e_2)$ is at least $g-2$. Introduce the same vertex on each of the edges $e_1$ and $e_2$ (subdivide each edge and identify the two new vertices) to obtain a new graph $\Gamma'$ of order $n+1$.
  \end{tech}

  Not surprisingly, these techniques have been previously used by many authors including \cite{jajcay2021spectra}. 
  They, however, did not gain much of a recognition, as they increase the order of the resulting graph which is, in the 
  original context of the Cage Problem, undesirable. In the context of our paper, it is interesting to ask \textit{"When are these techniques guaranteed to yield the entire $(k,g)$-spectrum (possibly starting from a single
  graph)?"} or  similarly,\textit{"What is the bound $N'(k,g)$ such that every element of the $(k,g)$-spectrum greater than $N'(k,g)$ can be obtained by these techniques?"}. Since sufficiently large graphs necessarily contain edges that
  are far apart, the second question, in particular, appears meaningful.

  \subsection{Edge Deletion and Vertex Addition Constructions}\label{tec2}
  The algorithms described in this and the following subsection are quite simple and we include
  their description for completeness. Since the graphs we considered were relatively small
  (several hundred vertices; at most), these algorithms were fast enough without a need for detailed optimization.
  
   We used two kinds of edge deletions, namely three-edge deletion and two-edge deletion. These two techniques exhibit similar algorithms. However, in three-edge deletion we delete three edges and add two vertices whereas in two-edge deletion we only delete two edges and add a vertex. We therefore use three-edge deletion to generate $(3,g)$-spectra and two-edge deletion to generate $(4,g)$-spectra (again, for the reasons of parity). We describe both techniques in the following algorithm description.

Algorithm \ref{DeleteEdgesAddVertices} takes four parameters for its input. The first parameter is a \lstinline{graph} object from \emph{GRAPE} package \cite{GRAPE} that we wish to modify (most of the
time a $(k,g)$-graph). The second parameter \lstinline{numberOfEdges} is the number of edges we intend to extract from the \lstinline{graph} object (three or two). The parameter \lstinline{numberOfVertices} is the number of vertices that need to be added to the original graph (two or one, respectively). The last parameter is the desired girth of the resulting graph.

First, in a for-cycle, the algorithm runs through all possible sets of edges of size \lstinline{numberOfEdges} in the input \lstinline{graph}. For each such combination of edges, we delete the edges from the \lstinline{graph} object with the function \lstinline{DeleteEdges()} and add new vertices according to the \lstinline{numberOfVertices} with the function \lstinline{AddVertices()}. At this moment, the algorithm constructed an irregular graph with isolated vertices and vertices of smaller degree than the degree $k$ of the original graph. With the function \lstinline{possibleEdges()}, we obtain the set of all edges that complete the graph into a regular graph and which do not cause a cycle of size smaller than \lstinline{girth}. The last function \lstinline{graphsByCombination()} goes through all subsets of \lstinline{possibleEdges} of required size and returns a list of graphs of the prescribed \lstinline{girth}. The for-loop stops when a graph with the desired parameters is found.

\begin{algorithm}[h] 
\caption{Obtaining graph with given girth by deleting edges and adding vertices}
\label{DeleteEdgesAddVertices}
\begin{algorithmic} 
\STATE \textbf{DeleteEdgesAddVertices(graph, numberOfEdges, numberOfVertices, girth)}

\FORALL{combinations of edges of size numberOfEdges}
	\STATE graph = DeleteEdges(graph, combinations of edges)
	\STATE graph = AddVertices(graph, numberOfVertices)
	\STATE possibleEdges := possibleEdges(graph, degree, girth)
	\STATE possibleGraphs := graphsByCombination(graph, possibleEdges, girth)
	\IF{possibleGraphs is not empty}
			\RETURN possibleGraphs
	\ENDIF
\ENDFOR
\end{algorithmic}
\end{algorithm}

   \subsection{Vertex Deletion Constructions}\label{tec4}
 Here, a $(k,g)$-graph of larger order is used as a starting graph, a few vertices are deleted and the remaining vertices are reconnected to form a new graph of smaller order. We used the following vertex deletions: one-vertex deletion, two-vertex deletion, three-vertex deletion, and four-vertex deletion. Our algorithms discard resulting graphs which are of shorter girths and/or of smaller degrees. 

Algorithm \ref{DeleteVertices} is of a very similar flavor to Algorithm \ref{DeleteEdgesAddVertices}. The difference between the two algorithms is that in the second algorithm no new vertices are added and vertices instead of edges are deleted. For its input, it takes three parameters. The parameters \lstinline{graph} and \lstinline{girth} are the same as in Algorithm \ref{DeleteEdgesAddVertices}. The parameter \lstinline{numberOfVertices} stands for the number of vertices we wish to remove from the input graph.

The for-loop goes through combinations of vertices of the given size \lstinline{numberOfVertices}. With each combination, it deletes the vertices with their adjacent edges using the function \lstinline{DeleteVertices()}. Afterwards, the algorithm continues in the same way as in Algorithm \ref{DeleteEdgesAddVertices}. The algorithm stops when it finds a graph of the desired degree and girth.

\begin{algorithm}[h] 
\caption{Obtaining graph with given girth by deleting vertices}
\label{DeleteVertices}
\begin{algorithmic} 
\STATE \textbf{DeleteVertices(graph, numberOfVertices, girth)}

\FORALL{combinations of vertices of size numberOfVertices}
	\STATE graph = DeleteVertices(graph, combinations of vertices)
	\STATE possibleEdges := possibleEdges(graph, degree, girth)
	\STATE possibleGraphs := graphsByCombination(graph, possibleEdges, girth)
	\IF{possibleGraphs is not empty}
			\RETURN possibleGraphs;
	\ENDIF
\ENDFOR
\end{algorithmic}
\end{algorithm}
 
  \subsection{Biggs's Tree Deletion Construction} \label{tec5}
  This technique is an implementation of Theorem~\ref{n5}. It is an extension of vertex deletion \ref{tec4}, but deletes an induced tree instead of deleting a random vertex or unrelated vertices. 
  The formula for determining the maximum order of a tree to be deleted is expressed in Theorem \ref{Biggs_tree_theorem}.  This technique has been used in generating graphs $Graph(3, 10, 102)$, and  $Graph(3,11,148;1)$.
 
  \begin{theorem}[\cite{biggs1998constructions}] \label{Biggs_tree_theorem}
 	Let $\Gamma$ be a cubic graph of girth $g \geq 4$ of order $n$ and let $r = \lfloor \frac{g}{4} \rfloor$. Then there exists a cubic graph $\Gamma'$ of order $n-\epsilon$ and girth $g-1$, where
 	\[ \epsilon =
 	\begin{cases} 
 		2^{r+1}-2 & \text{if } g \equiv 0,1 \pmod{4}  \\
 		3 \cdot 2^r-2 & \text{if } g \equiv 2,3 \pmod{4}
 	\end{cases}
 	\]
 \end{theorem}
 
Algorithm \ref{RemoveBiggsTree} based on the above theorem removes a subtree with a specified number of vertices and then reconnects the vertices adjacent to the subtree. It yields a new graph of girth smaller by one than the girth of the original graph. It takes the \lstinline{graph} in the \emph{GRAPE} package format \cite{GRAPE} as an input parameter.
In its first step, it computes the size of the excised subtree using the function \lstinline{BiggsTreeSize()} based on the girth of the given graph. Then, it inspects all vertices in the given graph. For each vertex $v$, it iteratively obtains a subtree of the computed size \lstinline{excisionSize} rooted at $v$.
After excising the subtree, it reconnects the `deficient' vertices and checks the resulting degree and girth in the function \lstinline{RemoveTree}. If the function \lstinline{RemoveTree} returns a graph of desired parameters, the algorithm stops.

\begin{algorithm}[h] 
\caption{Obtaining graph by removing the Biggs tree}
\label{RemoveBiggsTree}
\begin{algorithmic} 
\STATE \textbf{RemoveBiggsTree(graph)}
	\STATE excisionSize := BiggsTreeSize(girth);
	\FORALL{vertex in Vertices(graph)}
        \FORALL{subtree in IteratorOfSubtrees(graph, excisionSize)}
				\STATE newGraph := RemoveTree(graph, girth, subtree);
				\IF{newGraph is not empty}
					\RETURN newGraph;
                \ENDIF
		\ENDFOR
	\ENDFOR
\end{algorithmic}
\end{algorithm}

\subsection{Perfect Matching Deletion Construction}
All bipartite regular graphs contain a perfect matching. In general, removing a perfect matching from a $(k,g)$-graph results in a $(k-1,g')$-graph with $g' \geq g$; where $g'=g$ in case at least one of the original girth-cycles
is not affected by the perfect matching removal.
A particular example of this method involves deleting the edges of a perfect matching from a complete bipartite graph $K_{k,k}$ of degree $k$ and girth $4$. The resulting graph remains bipartite, with each vertex having degree $k-1$. 
In the case $k=3$, the girth $g'$ of the resulting graph increases to $6$, while the girth of the graph when $ k \geq 4 $  remains unchanged, as no new edges and therefore no shorter cycles were introduced, and at least one of the
original $4$-cycles was preserved. Using this approach, one can construct $(k-1)$-regular graphs of girth $4$ and order $2k=n(k-1,4)+2$, for all $ k \geq 4 $. Moreover, since the resulting graph is still regular and bipartite, this
method can be repeated, resulting in bipartite graphs of order $2k$ and degrees smaller than $k-1$. However, repeated application may result in increased girth and/or disconnected graphs.

 \subsection{Circulant Graph Construction and Its Variation}\label{circ1}
 A \emph{circulant graph} $G(n,S)$ is an arithmetically defined graph on $n$ vertices $u_0,u_1,\ldots,u_{n-1}$. Its adjacency is defined via a ($ \mathbb{ Z}_n $-generating) connecting set $S \subseteq \{0,1, 2,..., n-1\}$ closed under taking inverses, $ S = -S 
 \pmod{n}$, with two vertices $i$ and $j$ in $G(n,S)$ adjacent if and only if $i - j \pmod{n}$ belongs to $S$. We used the circulant graph construction to obtain
 elements of the $(4,4)$-spectrum, and altered the circulant graph construction a bit
 for the $(4,6)$-spectrum.
\smallskip
 
 \textbf{$(4,4)$-graphs:} To obtain tetrahedral graphs of orders $ n \geq 10 $, let $ S = \{ 1,n-1,3, n-3 \} $, i.e., for every vertex $u_i$,
 $i \in Z_n$, $u_i$ is adjacent to $u_{i+1}$, $u_{i-1}$, $u_{i+3}$, and $u_{i-3}$, with the indices calculated $ \pmod{n}$. This choice of the connecting sets yields $(4,4)$-graphs of orders $10, 11,12, \dots$.
\smallskip

 \textbf{$(4,6)$-graphs:} To obtain tetrahedral graphs of even orders and larger girth, we define the adjacency based on the parity of $i$:
 \begin{itemize}
 	\item $i$ is odd: $u_i$ is adjacent to $u_{i+1}$, $u_{i-1}$, $u_{i+7}$, and $u_{i+11}$,
 	\item $i$ is even: $u_i$ is adjacent to $u_{i+1}$, $u_{i-1}$, $u_{i-7}$, and $u_{i-11}$,
 \end{itemize}
with the indices calculated modulo $n$, $ n \geq 26 $ and even.
Even though the resulting graphs of girth $6$ are not necessarily circulants, the altered construction yielded $(4,6)$-graphs of all required even orders. The way to obtain the required odd orders has already been described in Section~\ref{sec:Comp3}.

\subsection{Group Divisible Generalized Petersen Graphs}\label{tec6}
Let $n\geq 3$ and $m\geq 2$ be positive integers such that $m$ divides $n$, let $a$ be a non-zero element of $\mathbb{Z}_m$, and let $K=(k_{0}, k_{1}, \dots, k_{m-1})$ be a sequence of elements from $\mathbb{Z}_n$ all of which are congruent to $a \pmod{m}$ and satisfy the requirement $k_j + k_{j-a} \not \equiv 0 \pmod n$, for all $j \in \mathbb{Z}_m$. \emph{ A Group Divisible Generalized Petersen graph $GDGP_m(n: k_{0}, k_{1}, \dots, k_{m-1})$} is a cubic graph with vertex set of order $2n$ consisting of the vertices $\{u_i, v_i \; | \; i \in \mathbb{Z}_n\}$, having
the edge set of size $3n$: $\{u_i,u_{i+1}\}$, $\{u_i,v_i\}$ and $\{v_{mi + j}, v_{mi + j + k_{j}}\}$, where $i \in \mathbb{Z}_{\frac{n}{m}}, j \in \mathbb{Z}_m$, where the arithmetic operations are performed $\pmod{n}$  \cite{jasenvcakova2020new}. The parameters of the cubic graphs of girths $8$, $10$, and $12$ constructed as GDGP-graphs are presented in
the respective subsections of the previous section.

\subsection{Canonical Double Cover Construction}\label{tec8}
 The following is a well-known recursive construction that preserves the degree of the original graph, and in the 
 special case of the canonical double cover, it doubles
 its order and its cycles of odd lengths. This means that it produces a $(k,g')$-graph from a $(k,g)$-graph of odd girth $g$, with $g'>g$ and even. For the sake of completeness, we present the construction, but only for the case when 
 the starting graph is already a $(k,g)$-graph. 
 
 Let $\Gamma$ be a finite $(k,g)$-graph, and let $D(\Gamma)$ denote the set of \textit{darts} of $\Gamma$, obtained by replacing each edge $e$ of $\Gamma$ with a pair of opposing darts (or arcs) $e$ and $e^{-1}$. A mapping $\alpha: D(\Gamma) \rightarrow G$ is called a \textit{voltage assignment} on $\Gamma$ if it satisfies the condition $\alpha(e^{-1}) = (\alpha(e))^{-1}$ for all $e \in D(\Gamma)$, where $G$ is a group referred to as the \textit{voltage group}. The \textit{voltage graph} (also known as the \textit{derived graph} or the \textit{lift}) of $\Gamma$ with respect to $\alpha$, denoted by $\Gamma^{\alpha}$, is a new graph with vertex set $V(\Gamma^{\alpha}) = V(\Gamma) \times G$ and edge set $E(\Gamma^{\alpha})$, where vertices $u_a$ and $v_b$ are adjacent in $\Gamma^{\alpha}$ if $e = (u,v) \in D(\Gamma)$ and $b = a \alpha(e)$.
\sloppy
A voltage graph $\Gamma^{\alpha}$ is a \textit{canonical double cover} of $\Gamma$ if the voltage group is $\mathbb{Z}_2$ and each edge of $\Gamma$ is assigned the non-zero voltage $1 \in \mathbb{Z}_2$. The canonical double cover construction is a particular type of voltage graph construction, which has been extensively studied by various authors (of the large number of articles considering the canonical double cover, consult, for example, \cite{eze2022recursive, gross2001topological}). In our work, we applied the canonical double cover construction to obtain the following graphs: $Graph(3,8,48;1)$, $Graph(3,12,224;1)$, $Graph(3,12,228)$, $Graph(4,8,134)$, $Graph(4,8,146)$, $Graph(5,6,60)$,  $Graph(6,4,14)$, $Graph(6,4,16)$, $Graph(6,4,18)$, $Graph(6,4,20)$, $Graph(6,4,22)$, $Graph(6,4,24)$, and $Graph(7,4,16)$.

As is well-known, a canonical double cover of a bipartite graph consists of two disconnected copies of the original graph. Thus, using the canonical double cover construction systematically, we also obtained the following graphs, which are all disconnected but were used as starting graphs with respect to erasing vertices and adding edges as described in Subsection~\ref{tec4}:
$(3,12)$-graph of order $252$,  $(5,4)$-graph of order $20$,
$(5,8)$-graph of order $304$, and $(7,4)$-graph of order $28$ (which consists of two disconnected copies of $K_{7,7}$).

\section{Concluding Remarks} \label{sec:Concl}
As already stated in the introduction, the aim of the research and results presented in our paper is to gain 
insights into the structure of $(k,g)$-graphs; with the graphs of small orders being of particular 
interest. We conclude by outlining possible directions for future research or applications.

Of the several open questions highlighted in \cite{exoo2012dynamic}, there is one that is attributed to several authors and which is based
on the simple observation that all known cages as well as record graphs of even girth happen to be bipartite.
This lead to the repeatedly stated conjecture that all even-girth cages must be bipartite. Whether the conjecture holds
true or not, it still leads to the following open problem we find both interesting and related to $(k,g)$-spectra:

\begin{center} \emph{For any given $k \geq 3 $ and \emph{even} $g \geq 4$, determine the smallest order $n$ such that 
there exists a $(k,g)$-graph of order $n$ which is not bipartite.}
    \end{center}

Our findings not only expand the current state of affairs with regard to the spectra of orders of $(k,g)$-graphs, but also highlight the challenges in determining complete spectra for larger values of $k$ and $g$. The gaps in our results, particularly for higher girths, underscore the complexity of this problem and suggest directions for future research in
finding new methods for constructing $(k,g)$-graphs.
The computational methods presented in our paper offer a practical approach to exploring $(k,g)$-spectra, complementing theoretical techniques, but should be viewed as just an introduction into the computational aspects of the considered
problems. The new found methods may not only bring advancement with regard to spectra of orders, but might prove
applicable to other related problems in Extremal Graph Theory.

Computational techniques and extensive use of computers constitute an inherent part of the approaches to the 
Cage Problem. The same is true for the determination of spectra of orders, and researchers who have employed
computational approaches to the Cage Problem certainly have developed a box of tools that should be easily 
adaptable to problems considered in here. In this context, we wish to point out one particular recent paper \cite{rui-big15} in which the authors employed methods similar to ours in that they
remove a tree from a graph of girth $8$ and subsequently reconnect the degree-deficient vertices
using integer programming solvers to obtain graphs of girth $7$. We see this approach as a much
more sophisticated version of those presented in our paper suitable for use with much larger graphs.

\nocite{*}
\bibliographystyle{plain}
\bibliography{main}
\end{document}